\documentclass[letterpaper, 12pt]{amsart}

\usepackage{amsmath, amsfonts, amssymb, amsthm, amsrefs, array,
 stmaryrd, graphicx, hyperref, mathrsfs, eucal, caption, soul, url}

\usepackage{tikz}
\usetikzlibrary{decorations.pathmorphing}

\usepackage[hhmmss]{datetime}

\allowdisplaybreaks
\AtBeginDocument{%
	\def\MR#1{}
}

\usepackage{color}

\theoremstyle{plain}
\newtheorem{thm}{Theorem}[section]
\newtheorem{lemma}{Lemma}[section]
\newtheorem{prop}[lemma]{Proposition}

\theoremstyle{definition}
\newtheorem{defn}[lemma]{Definition}
\theoremstyle{remark}
\newtheorem{remark}[lemma]{Remark}
\numberwithin{equation}{section}

\begin{document}

\title{Steady gradient Ricci Yang-Mills solitons on 2-orbifolds}

\author{Peng Lu}
\address{Department of Mathematics, University of Oregon, Eugene, OR 97403, USA}
\email{penglu@uoregon.edu}

\author{Jiuru Zhou}
\address{School of Mathematics, Yangzhou University, Yangzhou, 
Jiangsu 225002, China}
\email{zhoujiuru@yzu.edu.cn}
 
 \date{September 28, 2026}
 


\begin{abstract}
Following the recent work of M. Womack \cite{Wo26}
we consider the steady Ricci Yang-Mills solitons on 2d surfaces containing
 orbifold  point $\mathbb{R}^2/\mathbb{Z}_p$.
We show that there are a family of solitons depending parameter $\lambda$
which approaches to $\mathbb{Z}_p$-quotient of Hamilton cigar soliton
as  $\lambda  \to  (-2/p)^+$ and  approaches to (after rescaling) 
$\mathbb{Z}_p$-quotient of round sphere as $\lambda \to \infty$.
For any integer $q >p$ there is a $\lambda$ whose corresponding soliton 
metric is defined on football orbifold $S_{p,q}^2$.
\end{abstract}

\maketitle


\section{Introduction} \label{sec introd}

Ricci Yang-Mills flow was introduced by J. Streets in \cite{St07}.
The corresponding solitons are defined as the following.
Let $(M^n, g)$ be a Riemannian manifold.
Let $L \to M$ be a $U(1)$-line bundle with connection $A$ and curvature 2-form $F$.

 \begin{defn} 
 Let $f$ be a smooth function on $M$.
 A gradient Ricci Yang-Mills soliton is a tuple $(M^n,g,F, f)$ which satisfies
 \begin{equation} \label{eq steady grad Ricci ym soliton}
 \operatorname{Rc}(g) - \frac{1}{2} F^2 +\nabla^2 f = \hat{\lambda} g \quad
\text{~ and ~} \quad \operatorname{div}(e^{-f} F) =0,
 \end{equation}
 where $(F^2)_{ij} =g^{kl}F_{ik}F_{jl}$. When constant $\hat{\lambda} =0$, the soliton
 is called steady gradient Ricci Yang-Mills soliton. When $f$ is a constant, the soliton
 is called trivial. We may assume that $f(0) =0$ by a translation.
 \end{defn}
 
In the discussion below we allow $F$ to be a smooth 2-form, not necessarily a curvature
form of some $U(1)$-bundle.
In \cite{Wo26} Womack considers nontrivial 2d steady gradient Ricci Yang-Mills 
solitons and  finds a family of solutions which interpolate 
between Hamilton's cigar soliton and the round sphere. 
The solutions are proved to be rotationally symmetric.

Recall that the only 2d bad orbifolds are teardrop $S_p^2$ with  integer $p>1$ 
and  the football $S_{p,q}^2$  where $p$ and $q$  are two different positive 
integer \cite{Wu91}. 
Fix an integer $p>1$,
in this paper we consider 2d complete steady gradient Ricci Yang-Mills solitons 
which are defined on a surface $\Sigma^2$ which contains an orbifold point
of type  $\mathbb{R}^2 / \mathbb{Z}_p$. 
Note that the definition of the solitons can be easily extended to orbifolds. 
We will show that nontrivial such solitons must be 
rotationally symmetric and that there is a family of complete conical solitons $(g_{\lambda},
F_{\lambda}, f_{\lambda})$ such that for each integer $q>p$ there is a $\lambda$
such that $g_{\lambda}$ is an orbifold metric on $S_{p,q}^2$.
 More precisely we will prove

\begin{thm} \label{thm main of orbifold womack p version} 
For any fixed integer $p > 1$,
there is a 1-parameter family of complete rotationally symmetric steady gradient
Ricci Yang-Mills solitons on 2d surface $(\Sigma^2_{\lambda}, g_{\lambda},
F_{\lambda}, f_{\lambda}), \, \lambda \in [-2/p, \infty)$,
which have the following property.
The family $(\Sigma^2_{\lambda}, g_{\lambda})$ converges

\begin{enumerate}
 \item in the pointed Cheeger-Gromov sense to $2p$ times 
 the quotient of  Hamilton cigar soliton  metric on $\mathbb{R}^2/
\mathbb{Z}_p$ as $\lambda  \to  (-2/p)^+$;

 \item in the pointed Cheeger-Gromov sense to a cusp soliton as 
 $\lambda \to 0^-$;

 \item in the pointed Gromov-Hausdorff sense to a cusp soliton as 
 $\lambda  \to 0^+$;
 
 \item in the Gromov-Hausdorff sense, after a rescaling, to $S_{p,p}^2$ with the 
quotient round metric as $\lambda  \to \infty$;

 \item for each  integer $q >p$ the metric $g_{\lambda}$ is defined on orbifold 
 $S^2_{p,q}$ for $\lambda = \frac{2}{q-p}$.
\end{enumerate}
\end{thm}

Our interest of considering 2d steady gradient Ricci Yang-Mills solitons on orbifolds
comes from the study of Ricci flow on 2-orbifolds done by  L.F. Wu in \cite{Wu91}. Note that the existence of steady gradient Ricci Yang-Mills solitons on orbifold
has been proven by J. Streets in \cite[Proposition 6.5]{St19}, our contribution
is to make these solitons explicit and to observe some of their properties.

The paper is organized as follows. In Section~\ref{sec rot symm discussion}, 
we prove that a nontrivial gradient Ricci Yang--Mills soliton on a surface with 
an orbifold point is necessarily rotationally symmetric. 
In Section~\ref{sec reduction and solutions}
assuming rotational symmetry we reduce the steady soliton equations to an ODE 
system with parameter $\mu$
 and, after a suitable rescaling, we solve the equations in three cases 
$\mu=-1,0,1$, and obtain the corresponding families of soliton metrics.
 Finally, in Section~\ref{subsec pf main thm}, we analyze the limiting behavior 
 of these metrics and prove 
 Theorem~\ref{thm main of orbifold womack p version}, along with
 Theorem~\ref{thm uniqueness main of Ricci YM}  about
  the classification of nontrivial complete steady Ricci Yang--Mills solitons 
on orbifolds with at least one nontrivial  orbifold point.

\noindent
\subsection*{\rm{Acknowledgment}}
P.L. thanks the School of Mathematics, Yangzhou University for the
hospitality during his visit in the summer of 2026.
 J.R.Z. thanks Prof. Xingwang Xu and Hongyu Wang for their continual  support. We thank Prof. J. Streets for bringing our attention to his work on the existence of Ricci Yang-Mills solitons.

\section{Rotational symmetry of gradient Ricci Yang-Mills solitons 
on 2d orbifolds} \label{sec rot symm discussion}

We have the following

\begin{prop}\label{prop killing vf}
	Let $(\Sigma^2, g, F, f)$ be a complete 
	nontrivial gradient Ricci Yang--Mills soliton on
	a 2d surfaces which contains an orbifold point of type 
	$\mathbb{R}^2 / \mathbb{Z}_p$ for some integer $p>1$, and let $J$ 
	be the standard (orbifold) complex structure defined by $g$. 
	Then $J(\nabla f)$ is a nontrivial 
	Killing field on $\Sigma$. Moreover, $J(\nabla f)$ vanishes at the singular
point $\{0 \} /\mathbb{Z}_p$ and
	$(\Sigma, g, F, f)$ is rotationally symmetric.
\end{prop}

\begin{proof} That $J(\nabla f)$ is a nontrivial  Killing field follows from the proof
of \cite{Wo26}*{Proposition 2.2} directly. Because of the existence of nontrivial
orbifold point $\{0 \} /\mathbb{Z}_p$, the statement that $f$ must have a 
critical point can be proved much simpler than that in  \cite{Wo26}*{\S 5}.

To see the statement, at some neighborhood of point  $\{0 \} /\mathbb{Z}_p$
 the lift $\nabla \tilde{f}$ of $\nabla f$ to the uniformizing chart $\tilde V$ is a 
 $\mathbb{Z}_p$-invariant vector field. 
 The differential of the $\mathbb{Z}_p$-action on the tangent space
	$T_0\tilde V\cong \mathbb{C}$ is rotation by $e^{2\pi i/p}$, which fixes no nonzero
	vector, so $\nabla\tilde f(0)=0$.  Hence $J(\nabla f)(\{0 \} /\mathbb{Z}_p)=0$. 
	It follows that the Killing vector field  $J(\nabla f)$ generates
the required rotational symmetry for the gradient Ricci Yang-Mills solitons. 
\end{proof}

\section{Reduction and solitons metrics} \label{sec reduction and solutions}

\subsection{Reduction of equation (\ref{eq steady grad Ricci ym soliton})
when \texorpdfstring{$\hat{\lambda} =0$}{lambda-hat = 0}}
Consider steady gradient rotationally symmetric
 Ricci Yang-Mills solitons  on $(0,b) \times S^1$,
we write the metric as
\begin{equation}\label{eq metric rot symmetric RYM general}
g =dr^2 + \varphi^2(r) d \theta^2, \qquad \theta \in [0, 2\pi]/ \sim,
\end{equation}
$\varphi(r) >0$ on $(0,b)$. From \cite{Wo26}*{Lemma 3.1, Corollary 3.3} 
we have the following 
reduction of equations (\ref{eq steady grad Ricci ym soliton})
when $\hat{\lambda} =0$
\begin{align}
& \varphi^{\prime \prime}  -\frac{3 \mu}{2 } ( \varphi^2)^{\prime}
+ \mu^2 \varphi^3 + \lambda \varphi =0,  \label{eq varphi satisfy 2nd order}  \\
& f^{\prime} =\mu \varphi, \label{eq f reduction}   \\
& f^{\prime \prime} - \frac{1}{2} (f^{\prime})^2 + \frac{ \eta^2}{4} e^{2f}  = \frac{1}{2} \lambda,  \label{eq f 2nd order ode A} \\
&F = \psi dr \wedge d \theta,    \label{eq F via psi} \\
& \psi = \eta \varphi e^f, \label{eq psi reduction} \\
& \lambda = 2 \left ( \frac{ \eta^2}{4} + \frac{\mu}{p} \right ),  \label{eq lambda any mu}
\end{align}
where $\lambda \in [ 2\mu/p, \infty)$ is the parameter of steady  
Ricci Yang-Mills solitons. Constant $\mu$ can be fixed by rescaling 
in \S \ref{subsec rescling alpha}.

Fix  some integer $p>1$.
For the solution to extend to $r=0$ with orbifold type $\mathbb{R}^2 / \mathbb{Z}_p$
we require the initial condition  
\begin{equation} \label{eq varphi initial any mu}
\varphi(0) =0, \quad \varphi^{\prime}(0) = \frac{1}{p}.
\end{equation}

\subsection{Rescaling steady Ricci Yang-Mills soliton  to have \texorpdfstring{$\mu =-1, \, 0, \, 1$}{mu = -1, 0, 1}}
\label{subsec rescling alpha}
We may rescale steady soliton $(\Sigma, g, F, f)$ to another soliton
 $(\Sigma, \tilde{ g}, \tilde{F},
\tilde{f} )$ by some constant $\alpha >0$. From \cite{Wo26}*{Lemma 3.6} we have
\begin{align*}
& \tilde{\mu} = \alpha^{-1} \mu,  \quad \tilde{r} = \sqrt{\alpha} r, \\
& \tilde{\varphi}(\tilde{r}) = \sqrt{\alpha} \varphi(r),  \quad 
\tilde{f} (\tilde{r} ) = f(r), \quad \tilde{\psi} (\tilde{r} ) = \psi(r), \\
& \tilde{\eta } = \frac{1}{ \sqrt{\alpha}} \eta, \quad  
\tilde{\lambda} =\frac{\lambda}{\alpha}, 
\quad d \mu_{\tilde{g}} = \alpha d \mu_g. 
\end{align*}
Then equation (\ref{eq varphi satisfy 2nd order}) becomes 

\begin{equation}\label{eq varphi satisfy 2nd order rescaled}
\frac{d^2 \tilde{\varphi}}{d \tilde{r}^2}  - \frac{3 \tilde{\mu}}{ 2 } 
 \frac{d \tilde{\varphi}^2}{d \tilde{r} }  +  \tilde{\mu}^2 \tilde{\varphi}^3 
 +   \tilde{\lambda}\tilde{\varphi} =0.
\end{equation}
If we choose $\alpha =| \mu|$ when $\mu \neq 0$, then $\mu =-1, 1, 0$ 
after re-scaling.

Next we solve the soliton equation  (\ref{eq varphi satisfy 2nd order})
for each  $\mu =-1, 1, 0$ cases.

\subsection{\texorpdfstring{$\mu = - 1$}{mu = -1} case}\label{subsec mu -1}
In this case  we get from  (\ref{eq varphi satisfy 2nd order})
 (compare to \cite{Wo26}*{(3.3)})
\begin{equation} \label{eq varphi satisfy 2nd order mu -1}
\varphi^{\prime \prime} + \frac{3}{2} ( \varphi^2)^{\prime}
+ \varphi^3 + \lambda \varphi = 0,
\end{equation}
with  initial condition  (\ref{eq varphi initial any mu}).
For fixed integer $p>1$, this is a second order equation with one free parameter 
$\lambda = 2 \left ( \frac{ \eta^2}{4} - \frac{1}{p} \right )$.

\vskip .2cm
In case where $\lambda >0$, $\mu =-1$ and $\lambda = 2 \left (
\frac{\eta^2}{4} -\frac{1}{p} \right )$ we use the following 
method to solve\footnote{We use AI system DeepSeek to find the 
solutions} the ODE (\ref{eq varphi satisfy 2nd order mu -1}) 
with initial condition (\ref{eq varphi initial any mu}).
Let $\varphi = \frac{u^{\prime}}{u}$, then $u$ satisfies $u^{\prime \prime \prime}
+ \lambda u^{\prime}  = 0$  and we may choose the corresponding initial condition
$u(0) =1, \, u^{\prime}(0) =0, \, u^{\prime \prime}(0) = \frac{1}{p}$.
Hence we get solution
\begin{align}
& u^{\prime} = \frac{1}{p \sqrt{\lambda}} \sin (\sqrt{\lambda} r), \notag \\
& u = \frac{1}{p \lambda} ( p \lambda+1-  \cos (\sqrt{\lambda} r)), \notag \\
& \varphi=\varphi_{\lambda} (r)= \frac{ \sqrt{\lambda} \sin (\sqrt{\lambda} r)}{ p
 \lambda+1
-  \cos (\sqrt{\lambda} r))}.  \label{sol varphi lambda positive}
\end{align}
Hence $r_{\lambda,p} = \frac{\pi}{\sqrt{\lambda}}$ is the smallest positive $r$
such that $\varphi_{\lambda} (r) =0$.  The underlying surface is compact.
We have $\varphi_{\lambda}^{\prime}(r_{\lambda,p} ) 
= - \frac{\lambda}{  p \lambda+2}$
which is a monotone decreasing function of $\lambda$.
Given positive integer $p$, then range $\{ \varphi_{\lambda}^{\prime}(r_{\lambda,p} ), \,
\lambda \in (0, \infty) \} = \left  ( - \frac{1}{p}, 0 \right )$.
Hence for any integer $q > p$ we can find $\lambda_{p,q} = \frac{2}{q-p} $ such that 
$ \varphi_{\lambda_{p,q}}^{\prime}(r_{\lambda_{p,q},p} ) =-\frac{1}{q}$.
Hence the metric $dr^2 +\varphi_{\lambda_{p,q}}(r)^2 d \theta^2$ is a 
steady Ricci Yang-Mills solion defined on orbifold $S^2_{p,q}$.
By some further calculations one can show that the extended metric is real
analytic on $S^2_{p,q}$.

\vskip .2cm
In case where $\lambda \in (-2/p, 0), \, \mu =-1$ and $\lambda =2 \left ( \frac{\eta^2}{4} - \frac{1}{p} \right ) \in \left ( -\frac{2}{p}, 0 \right )$ it follows from
  \cite{Wo26}*{Theorem 3.21} that the solitons are complete  and noncompact.
Furthermore a simple calculation  using L'Hospital rule as 
in \cite{Wo26}*{Proposition 3.22}  gives the Gauss curvature
 at point $\{ 0 \}/ \mathbb{Z}_p$
\[
K(\{ 0 \}/ \mathbb{Z}_p ) = -\lim_{r \to 0^+} \frac{ \varphi_{\lambda} 
^{\prime \prime}(r)}{  \varphi_{\lambda}  (r)}= \frac{3}{p} + \lambda.
\] 
Actually solving the ODE as $\lambda >0$ case we have solution
for $\lambda  \in (-2/p, 0)$
\begin{equation}
\varphi (r) = \varphi_{\lambda} (r) =\frac{ \sqrt{|\lambda|} \sinh (\sqrt{|\lambda|} r)}{ p
| \lambda| - 1 +  \cosh (\sqrt{|\lambda|} r))}, 
\qquad r \in [0, \infty).  \label{sol varphi lambda negative}
\end{equation}

\vskip .2cm
In case where $\lambda =0 , \, \mu =  -1$ and $\eta = 2/ \sqrt{p}$ we have
from solving (\ref{eq varphi satisfy 2nd order mu -1})

\begin{equation} \label{sol varphi lambda zero}
\varphi =   \frac{r}{p+\frac{r^2}{2}}.
\end{equation}
The corresponding metric is a cusp soliton.

\vskip .2cm
In case where $\lambda  = -2/p$,  $ \mu =-1$ and $\eta = 0$,
from solving (\ref{eq varphi satisfy 2nd order mu -1})
we get solution
 \[
 \varphi(r) =\sqrt{\frac{2}{p}} \, \tanh \frac{r}{\sqrt{2p}}.
 \]
Let $s =\frac{r}{\sqrt{2p}} $, then metric $g = 2p (ds^2 + \frac{1}{p^2} 
\tanh^2 s d \theta^2 )$. Note that metric $ds^2 + \frac{1}{p^2} 
\tanh^2 s d \theta^2 $ is the quotient of the standard cigar metric on 
$\mathbb{R}^2/ \mathbb{Z}_p$.

\subsection{\texorpdfstring{$\mu =1$}{mu = 1} case} \label{subsec mu =one soliton sol}
In this case  we get from  (\ref{eq varphi satisfy 2nd order})

\begin{equation} \label{eq varphi satisfy 2nd order mu positive 1}
 \varphi^{\prime \prime}  -\frac{3 }{2 } ( \varphi^2)^{\prime}
+  \varphi^3 + \lambda \varphi =0,
\end{equation}
with initial condition  (\ref{eq varphi initial any mu}).
We have  from (\ref{eq lambda any mu})
\[
\lambda = 2 \left ( \frac{ \eta^2}{4} + \frac{1}{p} \right ) \geq \frac{2}{p}.
\]

Similar to solving equation (\ref{eq varphi satisfy 2nd order mu -1})
 we have solution 
\begin{equation} \label{eq  varphi mu equ one}
\varphi(r)= \varphi_{\lambda}(r)
 = \frac{ \sqrt{\lambda} \sin ( \sqrt{\lambda}  r)} { p\lambda -1
+ \cos ( \sqrt{\lambda}  r)}, \qquad  0 \leq r \leq \frac{ \pi}{  \sqrt{\lambda}}.
\end{equation}
From 
\[
\varphi_{\lambda}^{\prime}(r) = \frac{ \lambda( 1+ (p \lambda -1) 
\cos ( \sqrt{\lambda}  r) )} 
{ (p\lambda -1 + \cos ( \sqrt{\lambda}  r) )^2}
\]
we get
\[
 \varphi_{\lambda}^{\prime}( \frac{ \pi}{  \sqrt{\lambda}} ) 
 =  \frac{ \lambda }{ 2-p  \lambda}.
\]

For any $1 \leq q <p$ when $\lambda = \frac{2}{p-q}$ we have 
$ \varphi_{\lambda}^{\prime}( \frac{ \pi}{  \sqrt{\lambda}} ) =- \frac{1}{q}$. 
Hence the corresponding
metric is defined on $S_{p,q}^2$.

In the special situation of $\lambda =2/p$ we have
\[
\varphi(r) = \frac{ \sqrt{\lambda} \sin ( \sqrt{\lambda}  r)} { 1
+ \cos ( \sqrt{\lambda}  r)}, \qquad  0 \leq r \leq \frac{ \pi}{  \sqrt{\lambda}}.
\]
We have $\lim_{ r \to\frac{ \pi}{  \sqrt{\lambda}}} \varphi (r  ) = + \infty $.
The corresponding metric is incomplete.

\subsection{\texorpdfstring{$\mu =0$}{mu = 0} case}   \label{subsec mu =0 soliton sol}
In this case we have $f =0$ from (\ref{eq f reduction}) and the soliton is trivial.
We get from  (\ref{eq varphi satisfy 2nd order})
\begin{equation} \label{eq varphi satisfy 2nd order mu zero}
 \varphi^{\prime \prime}   + \lambda \varphi =0, \qquad 
 \varphi(0) =0, \quad \varphi^{\prime}(0) = \frac{1}{p}, 
\end{equation}
where $ \lambda = \frac{\eta^2}{2}$ with $\eta \in [ 0, \infty)$.
If $\eta \in ( 0, \infty)$, 
the solutions are $\varphi(r) = \frac{1} { \frac{1}{\sqrt{2}}\eta p} \sin  \left (
\frac{1}{\sqrt{2}}\eta r \right )$, this is a round sphere
of radius  $ \frac{\sqrt{2}} {\eta }$ quotient by a standard $\mathbb{Z}_p$-action. 
We also have $\varphi^{\prime} (\frac{\sqrt{2} \pi}{\eta} ) =- \frac{1}{p} $.
If $\eta = 0$, then $\lambda = 0$ and $\varphi^{\prime \prime} (r) = 0$,
the solution is $\varphi (r) = r/p$. Hence, this is the flat cone $\mathbb{R}^2/
\mathbb{Z}_p$.

Hence when $\mu =0$ the family of steady solitons in $\lambda$ (in turn 
$\eta \in (0, \infty)$) is asymptotic to the flat cone
$\mathbb{R}^2 / \mathbb{Z}_p$ when $\eta \to 0^+$ and to a round point
 quotient by $\mathbb{Z}_p$-action when $\eta \to \infty$.

\section{The proof of Theorem \ref{thm main of orbifold womack p version}}
\label{subsec pf main thm}

Assume $\mu =-1$, note that from solution $\varphi_{\lambda}$ in \S
\ref{subsec mu -1} the steady soliton $g_{\lambda}
=dr^2 +\varphi_{\lambda}(r)^2 d \theta^2$  are asymptotic in pointed 
Cheeger-Gromov topology to soliton
 $g_{ -2/p}$ when  $\lambda \to (-2/p)^+$ and 
  to solitons at $ g_0$ as $\lambda \to 0^-$ or $\lambda \to 0^+$.

Note that as in \cite{Wo26}*{Prop. 3.15} our soliton $g_{\lambda}$ also
verifies that $\lambda =0$ is a phase transition from 
noncompact  solitons and compact ones.

Finally we consider the asymptotics of the steady solitons $g_{\lambda}$
as $\lambda \to \infty$.
Suppose we have a family of solution $\varphi_i = \varphi_{\lambda_i}$ 
corresponding to $\mu =-1$
and $\lambda_i \to \infty$ we take rescaling with $\alpha_i =\lambda_i$
in \S \ref{subsec rescling alpha},
then the corresponding rescaled $ \tilde{\lambda}_i =1, \,
\tilde{\mu}_i = (\lambda_i)^{-1} \cdot (-1) 
\to 0$, hence the limit of $\tilde{\varphi}_i \to \tilde{\varphi}_{\infty}$ satisfies
equation
$ \frac{d^2}{d \tilde{r}^2}  \tilde{\varphi}_{\infty} +\tilde{\varphi}_{\infty} =0$. 
We have the initial condition (using base point at $r =0$ for each $i$)
$\tilde{\varphi}_{\infty}  (0) =0, \, \tilde{\varphi}'_{\infty}(0) =1/p$, hence
the solution is $\tilde{\varphi}_{\infty} (\tilde{r}) = \frac{1}{p} \sin \tilde{r}$,
which corresponds to  the quotient round metric on $S_{p,p}^2$.
Now Theorem \ref{thm main of orbifold womack p version} is proved.

\vskip .2cm
A simple consequence of  Proposition \ref{prop killing vf},
 Theorem \ref{thm main of orbifold womack p version}
  is  the following classification theorem (compare  
to \cite{Wo26}*{Theorem 5.6} and \cite{St10}*{Proposition 17})

\begin{thm} \label{thm uniqueness main of Ricci YM}
 Any complete  steady Ricci Yang-Mills soliton on a surface
 containing nontrivial orbifold point is isometric to
 up to rescaling (i) a member of the family in Theorem \ref{thm main of orbifold 
 womack p version} when $\mu <0$, (ii) a member in \S 
 \ref{subsec mu =one soliton sol} when $\mu >0$, 
 (iii) a flat  orbifold $\mathbb R^2/\Gamma$ for a discrete group $\Gamma$ of Euclidean isometries,
 or (iv) $S^2_{\sqrt{2}/|c| }/\Gamma$ for a finite subgroup
 $\Gamma\subset SO(3)$.  
\end{thm}

\begin{proof} 
If the solition is nontrivial, by combining Proposition \ref{prop killing vf},
 Theorem \ref{thm main of orbifold womack p version} and the analysis in  \S 
 \ref{subsec mu =one soliton sol}  and  \S \ref{subsec mu =0 soliton sol}, we get (i) and (ii).
 
 \vskip .1cm
 If the soliton is trivial, the steady soliton equation (\ref{eq steady grad Ricci ym soliton})
 becomes $\operatorname{div}F=0$ and $ \operatorname{Rc}=\frac{1}{2}F^2$,
 hence $F =c d \mu_g, \operatorname{Rc}=\frac{c^2}{2} g$, and the Gauss
 curvature $K=\frac{c^2}{2} $  for some constant $c$. 
 The space-form theorem for complete orbifolds of  constant-curvature 
 (\cite[Chapter 13]{Ra06}) gives:
 
		\begin{enumerate}
\item If $c =0$,  then $K=0$ and $M$ is isometric, as an orbifold, to
$\mathbb R^2/\Gamma$ for a discrete group $\Gamma$ of Euclidean isometries, 
which is not necessarily rotationally symmetric. This corresponds to (iii).
			
\item If $c \neq 0$, then $K=c^2/2 >0$ and $M$ is isometric,
			as an orbifold, to $S^2_{\sqrt{2}/|c| }/\Gamma$ for a finite subgroup
			$\Gamma\subset SO(3)$. Here $S^2_{\sqrt{2}/|c|}$ is the round sphere of
			radius $\sqrt{2}/|c|$. This corresponds to (iv).
		\end{enumerate}
\end{proof}

\begin{remark}
	The previous classification deals with $F$ to be orbifold 2-form. If one further 
	requires $F$ to be the curvature of a
	connection on $U(1)$-orbibundle, there will be some restrictions as follows \cite[Section~2]{MOY97}. For   a compact connected oriented two-orbifold $\Sigma$,
	denote its cone orders by $m_1,\ldots,m_j$ and put
	$\ell=\mathrm{lcm}(m_1,\ldots,m_j)$, then a smooth real orbifold two-form $F$ is the curvature
	of a unitary orbifold line bundle if and only if
		\begin{equation}
			\label{eq:degree-lattice}
			\frac1{2\pi}\int_\Sigma F\in\frac1\ell\mathbb{Z}.
		\end{equation}
	
	When $\alpha=1$. If $f=0$,  in the flat case $c=0$ there is no additional restriction: $F=0$
	 is realized by the trivial orbibundle with a flat connection.  
	 In the spherical case $c \not=0$, from the Chern-Weil theory on the lifting 
	  $S^2_{\sqrt{2}/|c| }$ we have $\int_{S^2_{\sqrt{2}/|c| }} c d\mu_g = 2 \pi n$ for 
	  some integer $n \neq 0$, we get $c = 4/n$.
	
	If $f$ is non-constant, for 2-form $F = \eta \varphi e^f dr \wedge d \theta$
	to be the curvature form of some $U(1)$-orbibundle we have the following
	conclusion, but we skip their argument.
	 None of teardrop solutions $S^2_{1,q}$ admits $F$ to be
	 orbibundle curvature. For the compact football solutions $S^2_{p,q}$,
	$F$ is the curvature form if and only if
	\[
	pq|p-q|
	\]
	is a perfect square.
\end{remark}


\end{document}